\documentclass[12pt,a4paper]{article}
\usepackage{hyperref}
\usepackage{simon}
\usepackage[thmmarks,amsmath]{ntheorem}

\theoremnumbering{arabic}
\theoremstyle{plain}
\RequirePackage{latexsym}

\theorembodyfont{\itshape}
\theoremheaderfont{\normalfont\bfseries}
\theoremseparator{}
\theoremsymbol{}
\newtheorem{theorem}{Theorem}
\newtheorem{proposition}{Proposition}
\newtheorem{lemma}{Lemma}

\theoremheaderfont{\normalfont\itshape}
\theorembodyfont{\upshape}
\newtheorem{example}{Example}
\newtheorem{definition}{Definition}

\theoremstyle{nonumberplain}
\theoremheaderfont{\scshape}
\theorembodyfont{\normalfont}
\theoremsymbol{\ensuremath{\Box}}
\newtheorem{proof}{Proof}

\RequirePackage{amssymb}
\qedsymbol{\ensuremath{\Box}}
\theoremclass{LaTeX}

\usepackage{cdot}

\DeclareMathOperator{\auc}{AUC}
\DeclareMathOperator{\tauc}{tAUC}

\usepackage{bbm}
\newcommand{\ind}{\mathbbm{1}}

\newcommand\X{\mathcal{X}}
\newcommand\Y{\mathcal{Y}}

\newcommand\Family{\mathcal{F}}
\newcommand\Predspace{\mathcal{Z}}

\newcommand\power{\wp}

\usepackage{natbib}

  \title{\bf Empirical AUC for evaluating probabilistic forecasts}
  \author{Simon Byrne%
  \hspace{.2cm}\\
  Department of Statistical Science\\
  University College London\\
  United Kingdom\\
  \texttt{simon.byrne@ucl.ac.uk}}

\begin{document}
\maketitle

\begin{abstract}
  Scoring functions are used to evaluate and compare partially probabilistic
  forecasts. We investigate the use of rank-sum functions such as empirical
  Area Under the Curve (AUC), a widely-used measure of classification
  performance, as a scoring function for the prediction of probabilities of a
  set of binary outcomes. It is shown that the AUC is not generally a proper
  scoring function, that is, under certain circumstances it is possible to
  improve on the expected AUC by modifying the quoted probabilities from their
  true values. However with some restrictions, or with certain modifications,
  it can be made proper.
\end{abstract}
\noindent%

\section{Introduction}
\label{sec:introduction}

Predicting the outcomes of multiple binary variables is a common problem
across a variety of application domains, such as fraud detection, credit risk
evaluation, medical diagnostics and weather forecasting. Such forecasts
typically carry some information describing the uncertainty of the forecaster,
such as assigning explicit probabilities or some other numerical value to each
variable that allows the variables to be ranked in order of relative
probability of occurrence.

This paper investigates numerical measures for evaluating and comparing the
accuracy of such forecasts.  Although such measures have always been important
for comparing algorithms, their role has become increasingly important with
the popularity of prediction competitions, where it is necessary to precisely
quantify the performance of participants. In particular, we use the framework
of \emph{scoring functions}, which maps the prediction and subsequent
observation to a single real number, the \emph{score}, representing the reward
to the forecaster. The aim of the forecaster is then to maximise this reward.

Scoring functions can be viewed as extensions of \emph{scoring rules}
(section~\ref{sec:scoring-rules}), which require that the forecast be fully
probabilistic, providing a full joint probability distribution over the set of
all possible outcomes, which can be infeasible and unnecessary in many
situations. Scoring functions (section~\ref{sec:scoring-functions}) on the
other hand can make use of partial probabilistic information such as marginal
distributions, or rankings of expected values. One desirable feature of both
scoring rules and scoring functions is that they be \emph{proper}: that the
forecaster always has the incentive to be honest, in that the forecast which
maximises their expected score matches their true belief.

The focus of this paper is on a class of scoring functions termed
\emph{rank-sum functions} (section~\ref{sec:rank-sum}), the most well-known of
which is the \emph{area under the curve} (AUC), the curve in question being
the receiver operating characteristic (ROC). The ROC and AUC describe the
usefulness of the forecast in terms of its ability to discriminate between
positive and negative outcomes. Note that this paper specifically focuses on
the empirical AUC, and not the theoretical quantity that is perhaps more often
studied: this distinction is explained in detail in
section~\ref{sec:theoretical-auc}.

The main results (section~\ref{sec:proper-rank-sum}) identify sufficient
conditions for rank-sum scoring functions to be proper for evaluating the
accuracy of forecasts of the marginal probabilities of a sequence of binary
forecasts. In general, the AUC is not of this class, and a counter-example is
provided which demonstrates a case in which the AUC is \emph{not} a proper
scoring function, in that there exist distributions under which the forecaster
might improve their expected score by quoting probabilities different than
their true belief.

This framework can be further extended to the case where instead of making a
direct prediction, the forecaster is required to provide a mapping that
indirectly makes predictions from an as-yet unobserved covariate
(section~\ref{sec:map-scoring}). In section~\ref{sec:discussion}, we discuss
some open questions, and problems with extending the framework to a sequential
setting.

\section{Scoring of forecasts for binary outcomes}
\label{sec:scoring}

\subsection{Scoring rules}
\label{sec:scoring-rules}


Consider the setting where one is eliciting forecasts about some future
outcome $Y$ that takes values in an \emph{outcome space} $\Y$. A
\emph{probabilistic forecast} is a distribution $Q$ for $Y$ that
describes the forecasters uncertainty of $Y$. We define $\Family$ to be a
family of distributions over $\Y$ that are under consideration.

After the actual outcome $Y=y$ is observed, the reward to the forecaster is
determined by a \emph{scoring rule}, a function $S : \Y \times \Family \to
\R$, that maps the quoted $Q$ and observed outcome $y$ to a real number
$S(y,Q)$ termed the \emph{score}. We take scoring rules to be \emph{positively
  oriented}, that is the score represents the reward to the forecaster, who
therefore aims to maximise this quantity. In a decision theoretic context, the
negation of the score can be considered a \emph{loss
  function}. Mathematically, the problem can be precisely phrased in the form
of a game between a Forecaster and Nature \citep{dawid2012}.

For any $P \in \Family$, we can then define the \emph{expected score} as the
$\E_P[S(Y,Q)]$, where $Y$ is generated from $P$.  A scoring rule $S$ is
\emph{proper} if an optimal strategy for the forecaster is to quote a
distribution that matches their actual uncertainty, that is, if for all $Q,P
\in \Family$,
\begin{equation}
  \label{eq:proper-scoring-rule}
  \E_P[S(Y,Q)] \leq \E_P[S(Y,P)].
\end{equation}
Additionally, $S$ is termed \emph{strictly proper} if this is the only optimal
strategy, \ie \eqref{eq:proper-scoring-rule} is an equality only if
$Q=P$. Proper scoring rules for discrete variables have been extensively
studied \citep[\eg][]{dawid2012}; common examples include the Brier, spherical
and the log scores.

In this paper, we will consider the outcome space to be a vector of binary
variables,
\begin{equation*}
  Y = (Y_1,\ldots, Y_n) \in \Y = \{0,1\}^n .
\end{equation*}
In this case, the distribution $Q$ takes values on $\Delta_{2^n-1}$, the
$(2^n-1)$-dimensional unit simplex. If the family $\Family$ is the set of all such
distributions, then for large values of $n$ this can place a large burden in
terms of time and resources in constructing, communicating and evaluating the
score of the forecast. This motivates a more flexible framework.


\subsection{Scoring functions}
\label{sec:scoring-functions}

Suppose that instead of supplying a distribution $Q$ from a family $\Family$, we
require forecaster to quote a forecast from an arbitrary set $\Predspace$,
which we will term the \emph{prediction space}. Then a \emph{scoring function}
is a mapping of the form $s : \Y \times \Predspace \to \R$. \citet{gneiting2011}
extensively studied scoring functions in the context of point forecasts, where
$\Predspace = \Y$, though as we shall demonstrate, the concept extends directly to a more
general context.

The price of this generality is that we now need to explicitly specify the
aspects of the forecasters uncertainty that we want to capture. This can
be described by a \emph{(statistical) functional}, a possibly set-valued
function, $T : \Family \to \Predspace$ or $T : \Family \to \power \Predspace$,
where $\power \Predspace$ denotes the power set of $\Predspace$.

A scoring function $s$ is then said to be $T$-\emph{proper}
(\citet{gneiting2011} uses the term \emph{consistent})  if
for all $P \in \Family$, and all $u \in \Predspace$,
\begin{equation}
  \label{eq:exp-scorefn}
  \E_P[s(Y, u)] \leq  \E_P [s(Y, T(P))]
\end{equation}
for $\Predspace$-valued functional $T$, or for a set-valued functional $T$,
\begin{equation}
  \label{eq:exp-scorefn-set}
  \E_P[s(Y, u)] \leq  \E_P [s(Y, t)] \quad \text{for all $t \in T(P)$}.
\end{equation}
Furthermore, we can define $s$ to be \emph{strictly $T$-proper} if equality
holds only if $u = T(P)$ or $u \in T(P)$, respectively. Note that the
condition in \eqref{eq:exp-scorefn-set} implies that for any proper scoring
function $s$ of a set-valued functional, the expected score $\E_P [s(Y,t)]$
must be constant for all $t \in T(P)$. As would be expected from the
terminology, there is a strong link between scoring functions and scoring
rules, in that a (strictly) proper scoring function defines a (strictly)
proper scoring rule \citep[Theorem 3]{gneiting2011}.

In this paper, we focus on two specific classes of functionals for
distributions on $\Y = \{0,1\}^n$.

\subsubsection{Marginal scoring}
\label{sec:marginal-scoring}

\begin{definition}
  The \emph{marginal functional} $M$ maps a joint distribution to the
  marginal probabilities of each element of $Y$,
  \begin{equation*}
    M(P) = \E_P[Y] = \bigl( P[Y_1 = 1], \ldots, P[Y_n = 1] \bigr).
  \end{equation*}
  This functional reduces the $(2^n-1)$-dimensional distribution space to the
  $n$-dimensional prediction space $\Predspace = [0,1]^n$.
\end{definition}

We can easily construct scoring functions for the marginal functional as
functions of scoring rules for the individual elements of $Y$.
\begin{theorem}
  Let $S_i : \{0,1\} \times [0,1] \to \R$ be a scoring rule for a single
  binary outcome, such as the logarithmic, quadratic or Brier score. Then the
  scoring function
  \begin{equation*}
    s(y,m) = \sum_{i=1}^n S_i(y_i,m_i)
  \end{equation*}
  is (strictly) $M$-proper if each of the $S_i$ are (strictly) proper.
\end{theorem}
\begin{proof}
  Each $S_i$ can be maximised independently by choosing $m_i = \E[Y_i]$.
\end{proof}

\subsubsection{Rank scoring}
\label{sec:rank-scoring}

Recall that a \emph{total preorder} is a transitive and reflexive relation
$\precsim$ such that for any pair $i,j$, at least one of $i \precsim j$ or $j
\precsim i$. Given such a $\precsim$, we can define $i \sim j$ as the
symmetric relation $i \precsim j$ and $i \succsim j$ and $i \prec j $ as the
asymmetric relation $i \not\succsim j$ (which due to totality, implies $i \precsim
j$). Note $\precsim$ also implies a total ordering of the equivalence
classes under $\sim$.

Define $\Xi_n$ to be the set of total preorders on the set of indices $I =
\{1,\ldots,n\}$, then any vector $v \in \R^n$ \emph{induces} an element of $\precsim_v
\in \Xi_n$ by
\begin{equation*}
  i \precsim j \quad \Leftrightarrow \quad v_i \leq v_j.
\end{equation*}

\begin{definition}
  The \emph{exact rank functional} $R:\Family \to \Xi_n$ maps a joint distribution to the
  total preorder induced by the marginal functional $M$.
\end{definition}

The exact rank functional can also be characterised in terms of pairwise
comparisons.
\begin{proposition}
  Let $\precsim = R(P)$ for some distribution $P$ on $\Y$. Then
  \begin{equation*}
    i \precsim j 
    \quad \Leftrightarrow \quad
    P[Y_i > Y_j] \leq P[Y_i < Y_j].
  \end{equation*}
\end{proposition}
\begin{proof}
  By adding $P[Y_i=1,Y_j=1]$ to both sides, we have that
  \begin{equation*}
    P[Y_i=1, Y_j=0] \leq P[Y_i=0, Y_j=1] 
    \quad \Leftrightarrow \quad
    P[Y_i=1] \leq P[Y_j=1]
  \end{equation*}
\end{proof}
In the case where all the elements of $M(P)$ are unique, $R(P)$ is a total
order. We define $\Omega_n \subseteq \Xi_n$ to be the set of all total orders
on $I$. 

Note that the exact rank functional requires that ties ($\E[Y_i] = \E[Y_j]$)
be identified exactly. We define a weaker notion under which the ties can be
ignored. A relation $\precsim'$ is \emph{contained} in a relation $\precsim$ if $\precsim'
\subseteq \precsim$, that is, if $i \precsim' j$ implies that $i \precsim j$.
\begin{definition}
  The \emph{weak rank functional} $R^*:\Family \to \power\Xi_n$ is the
  set-valued functional that maps a probability distribution to the set of
  total preorders contained in the exact rank functional:
  \begin{equation*}
    R^*(P) = \{\precsim \in \Xi_n \,:\, \precsim \subseteq R(P)\}.
  \end{equation*}
\end{definition}
As a result, if all elements of $M(P)$ are unique, then $R^*(P) = \{R(P)\}$,
and conversely if all the elements of $M(P)$ are equal, then $R^*(P) = \Xi_n$.

Given an $R^*$-proper scoring function $s$, we can construct a
$M$-proper scoring function $s'$, via $s'(y, m) = s(y, \precsim_m)$.
Of course, such a scoring function can never be strictly $M$-proper, as
$\precsim_m$ is preserved under any monotonic increasing transformation.

An advantage of rank-based scoring functions is that they allow the
use of more abstract measures of propensity other than probability, and make
it possible to compare forecasts generated by a wide variety of algorithms,
whose outputs need not necessarily have a direct probabilistic
interpretation. The downside is that we lose the ability to say anything about
the \emph{calibration} of the forecaster.


\section{Rank-sum scoring functions}
\label{sec:rank-sum}

We now consider a particular class of rank-based scoring functions. For any
total preorder $\precsim$, we define its \emph{rank vector} $\rho:\Xi_n \to
\R^n$ to be the net number of elements that precede each element,
\begin{equation*}
  \rho_i(\precsim) = \sum_{j=1}^n \ind_{j \precsim i} - \ind_{j \succsim i} 
\end{equation*}
We will consider the class \emph{rank-sum} scoring functions, of the form
\begin{equation}
  \label{eq:ranksum}
  s(y,\precsim) = g(y) + \sum_{i=1}^n \sigma_i(y) \rho_i(\precsim).
\end{equation}
for some functions $g$ and $\sigma = (\sigma_i)_{i=1,\ldots,n}$

\begin{example}[Wilcoxon--Mann--Whitney $u$]
  The most well-known example of such a function is the
  \emph{Wilcoxon--Mann--Whitney $u$}, commonly used as a nonparametric test
  statistic for comparing magnitude of two random variables. It is defined as
  the number of times observations where $y_i = 0$ precede observations where
  $y_i = 1$, with ties counting as half
  \begin{equation}
    \label{eq:wmw-indh}
    u(y,\precsim) = \sum_{i:y_i =0} \sum_{j:y_j =1} \ind_{i \prec j} +
    \tfrac{1}{2} \ind_{i \sim j} .
  \end{equation}
  The term inside the summation is equal to $\tfrac{1}{2}[1 + \ind_{i\precsim j}
  - \ind_{i \succsim j}]$, and so
  \begin{equation*}
    u(y,\precsim) = \tfrac{1}{2} n_0(y) n_1(y) + \tfrac{1}{2} \sum_{i,j=1}^n
    y_i (1 - y_j) (\ind_{i\precsim j}  - \ind_{i \succsim j}).
  \end{equation*}
  where $n_1(y) = \sum_{i=1}^n y_i$, and $n_0(y) = n-n_1(y)$. By symmetry, we
  have that $\sum_{i,j}  (\ind_{i\precsim j}  - \ind_{i \succsim j}) =0$, and hence,
  \begin{equation*}
    u(y,\precsim) = \tfrac{1}{2} n_0(y) n_1(y) + \tfrac{1}{2}\sum_{i=1}^n y_i \rho_i(\precsim).
  \end{equation*}
  For a fixed $y$, $u$ will take values on the half-integers
  $0,\tfrac{1}{2},1,\ldots,n_0(y) n_1(y)$.
\end{example}

\begin{example}[Area under the curve]
  \label{exm:auc}
  The \emph{receiver operating characteristic} (ROC) describes the trade-off
  of sensitivity and specificity (or type I and type II error) of a preorder, and is
  calculated by plotting the true positive rate against the false positive
  rate that would be obtained by taking different elements of the preorder as
  the cutoff.

  It can be described as the parametric curve on $[0,1] \times [0,1]$,
  starting at $(1,1)$, then linearly connecting the points
  \begin{equation}
    \label{eq:roc}
    \left( 
      \sum_{j: y_j = 0} \frac{\ind_{j \succ i}}{n_0(y)},  
      \sum_{j: y_j = 1} \frac{\ind_{j \succ i}}{n_1(y)}
    \right),
  \end{equation}
  for each equivalence class $i$ under $\sim$, in the order of $\prec$.
  
  The \emph{area under the curve} (AUC) is then the total area under this
  curve, which will take values on $[0,1]$. It is
  well-established \citep[\eg][]{hanley1982} that this is in fact equal to
  the Wilcoxon--Mann--Whitney $u$, standardised by dividing by $n_0(y)n_1(y)$.

  Note that if the outcomes are identical (\ie $y =\mathbf{0}$ or
  $\mathbf{1}$), then the ROC and AUC are not properly defined. For
  convenience, we can define the AUC to be $1/2$ in both these cases, however
  the choice of this constant does not affect any of the results other than
  Theorem~\ref{thm:exp-auc}.
  
  As a result, we can write
  \begin{equation*}
    \auc(y,\precsim) = \tfrac{1}{2} + \tfrac{1}{2} \sum_{i=1}^n \alpha_i(y)
    \rho_i(\precsim)
    \quad \text{where}\ 
    \alpha_i(y) =
    \begin{cases}
      \displaystyle
      \frac{y_i}{n_0(y) n_1(y)} & \quad n_1(y) \ne 0, n, \\
      0 & \quad \text{otherwise}.
    \end{cases}
  \end{equation*}
  Also related is the \emph{Gini coefficient}, $g(y, \precsim) =
  2\auc(y,\precsim) - 1$, which is twice the net area of the
  ROC above the diagonal, and takes values on $[-1,1]$.
\end{example}

\subsection{Relation to theoretical AUC}
\label{sec:theoretical-auc}

Although the AUC has been widely explored in the literature, much of this
work \citep[\eg][]{agarwal2005,clemencon2008,hand2009,flach2011} focuses on a
related but distinct quantity, which we will term the theoretical AUC.

Let $\theta$ be a joint distribution for a random pair $(X_i,Y_i)$, where
$X_i$, taking values in some set $\mathcal{X}_{\Cdot}$, is termed the
\emph{covariate} or \emph{feature}, and $Y_i$ is a single binary response. For
some mapping $f: \mathcal{X}_{\Cdot} \to \R$, we define the conditional CDFs
$F_y(z) = \theta[ f(X_i) < z \mid Y_i = y ]$. Then the \emph{theoretical
  ROC} replaces the empirical quantities of \eqref{eq:roc} with their
theoretical equivalents,
\begin{equation*}
  \bigl(1 - F_0(z), 1-F_1(z) \bigr), \quad z \in \R
\end{equation*}
which again, describes a curve over $[0,1] \times [0,1]$. Similarly, the
\emph{theoretical AUC}, denoted $\tauc(\theta, f)$, is the area under this curve.

The theoretical AUC can be rewritten as the conditional
expectation \citep[\eg][Proposition B.2]{clemencon2008},
\begin{equation}
  \label{eq:tauc-expect}
  \tauc(\theta, f) = \E \left[ \ind_{f(X_1) > f(X_2)} + \tfrac{1}{2} \ind_{f(X_1) =
      f(X_2)} \mid  Y_1 = 1, Y_2  = 0\right],
\end{equation}
where the expectation is with respect to the product measure of
$\theta\times\theta$ for $[(X_1,Y_1) , (X_2,Y_2)]$.

The relationship between the empirical and theoretical AUCs is
well-established, though for completeness we clarify the usual
presentation \citep[\eg][Lemma 2]{agarwal2005}.
\begin{theorem}
  \label{thm:exp-auc}
  Let the pairs $(X_1,Y_1), \ldots, (X_n,Y_n)$ be independent and
  identically distributed as $\theta$, then the expected empirical AUC,
  \begin{equation*}
    \E[\auc(Y,\precsim_{f(X)})] = (1-\pi_0^n -\pi_1^n) \tauc(\theta,f) + \tfrac{1}{2} (\pi_0^n +\pi_1^n)
  \end{equation*}
  where $\pi_c = \theta(Y_i = c)$.
\end{theorem}
\begin{proof}
  For any vector $y \ne \mathbf{0}, \mathbf{1}$, the expectation of
  \eqref{eq:wmw-indh} conditional on $Y=y$ gives an expression of the form of
  \eqref{eq:tauc-expect}, and hence $\E[\auc(Y ,\precsim_{f(X)}) \mid Y=y]
  =\tauc(\theta, f)$.
\end{proof}

We emphasise several key differences between the empirical and theoretical
AUC. Firstly, the theoretical AUC is a function of the mapping $f$ from $X_i$
that is used to induce a ranking on $Y_i$ (confusingly, this is itself
referred to as a ``scoring function'' in the literature).

Another distinction is that the distribution $\theta$ is now a
hypothetical sampling model for a single pair $(X_i,Y_i)$, whereas the
previous distribution $P$ describes the forecasters uncertainty for a set
$(Y_1,\ldots,Y_n)$. We emphasise that these are distinct concepts:
whereas the i.i.d. assumption is typically reasonable in a sampling context,
it is extremely unrealistic for describing uncertainty, in that it would imply
that there is absolutely no information to be gained about $Y_n$ from the
other $Y_1,\ldots,Y_{n-1}$.

Additionally, although the negation of $\tauc(\theta,f)$ can still be interpreted as a loss
function in the standard decision-theoretic sense (\eg for deriving minimax
procedures), $\tauc(\theta,f)$ cannot be used as a scoring function as
$\theta$ is typically never observed directly.


\subsection{Proper rank-sum scoring functions}
\label{sec:proper-rank-sum}

To determine the propriety of such scoring functions, we utilise the following
key lemma.
\begin{lemma}
  \label{lem:sumvec}
  For any fixed vector $v \in \R^n$, the quantity
  \begin{equation}
    \label{eq:sumvec}
    \sum_{i=1}^n v_i \rho_i(\precsim)
  \end{equation}
  is maximised over $\precsim \in \Xi_n$ if and only if $\precsim$ is contained
  in $\precsim^{(v)}$, the preorder induced by $v$.
\end{lemma}
\begin{proof}
  Firstly, note that if we were to consider only total orders $\precsim \in
  \Omega_n$, then the statement is a direct result of the rearrangement
  inequality.  For any total preorder $\precsim \in \Xi_n$, define
  $A(\precsim)$ to be the set of total orders contained in $\precsim$, that is
  $A(\precsim) = R^*(\precsim) \cap \Omega_n$. Then for any $i,j$, by symmetry
  we have that
  \begin{equation*}
    \ind_{i \precsim j} = \frac{1}{|A(\precsim)|} \sum_{\precsim' \in A(\precsim)} \ind_{i \precsim' j}.
  \end{equation*}
  Therefore $\rho(\precsim)$ is the average of all $\rho(\precsim')$ for
  $\precsim' \in A(\precsim)$. It follows then that \eqref{eq:sumvec} is is
  maximised if and only if all such $\precsim'$ are themselves contained
  $\precsim^{(v)}$, which in turn implies that $\precsim$ itself is contained in
  $\precsim^{(v)}$.
\end{proof}

This then leads to our main result.
\begin{theorem}
  A rank-sum scoring function $s$ of the form in \eqref{eq:ranksum} is
  strictly $R^*$-proper if and only if $\precsim_{Pf}$, the preorder induced
  by $\E_P[\sigma_i(Y)]$, is an element of $R^*(P)$ for all $P \in \Family$.
\end{theorem}
\begin{proof}
  By the linearity of expectation, we have that
  \begin{equation*}
    \E_P[s(Y,\precsim)] = \E_P[g(Y)] + \sum_{i=1}^n \E_P[\sigma_i(Y)]
    \rho_i(\precsim).
  \end{equation*}
  By Lemma~\ref{lem:sumvec}, this can be maximised by any $\precsim$ contained
  in $\precsim_{Pf}$. These are all elements of $R^*(P)$ if and only if
  $\precsim_{Pf}$ itself is in $R^*(P)$.
\end{proof}

Consequently, the Wilcoxon--Mann--Whitney $u$ function is a strictly
$R^*$-proper scoring function, however the same cannot be said of the AUC.
\begin{example}
  \label{exm:auc-counter}
  Define the distribution $P$ on $(Y_1, Y_2, Y_3, Y_4)$ with the following
  non-zero probabilities:
  \begin{equation*}
    P(1,1,0,0) = \tfrac{1}{2}, \quad
    P(0,0,1,0) = \tfrac{7}{16}, \quad
    P(0,0,0,1) = \tfrac{1}{16}.
  \end{equation*}
  Then defining $\alpha$ as in Example~\ref{exm:auc}, we have that
  \begin{equation*}
    \E[Y] = \left( 
      \tfrac{1}{2}, \tfrac{1}{2}, \tfrac{7}{16}, \tfrac{1}{16}
    \right)
    \quad \text{and}\quad
    \E[\alpha(Y)] = \left( 
      \tfrac{1}{8}, \tfrac{1}{8}, \tfrac{7}{48}, \tfrac{1}{48} 
    \right).
  \end{equation*}
  Define $\precsim_P$ and $\precsim_{\alpha}$ as the preorders induced by
  $\E[Y]$ and $\E[\alpha(Y)]$, respectively. Then $\rho(\precsim_P) =
  (2,2,-1,-3)$ and $\rho(\precsim_{\alpha}) = (0,0,3,-3)$, with expected AUCs
  \begin{equation*}
    \E[\auc(Y,\precsim_P)] = \tfrac{31}{48} < \E[\auc(Y,\precsim_{\alpha})] = \tfrac{33}{48}.
  \end{equation*}
\end{example}
This rather contrived example is illustrative of how the problem arises,
namely the denominator of $\alpha$ can alter the relative importance of
certain outcomes. Nevertheless, there exist certain families $\Family$ under
which AUC is indeed proper.

\begin{theorem}
  \label{thm:known-sum}
  If the number of positive outcomes $n_1(Y)$ is almost surely constant for
  all $P \in \Family$, then AUC is a strictly $R^*$-proper scoring function.
\end{theorem}
\begin{proof}
  If $n_1(Y) = r$ almost surely, then $\E_P[\alpha_i(Y)] = \E_P[Y_i] / \bigl((n-r)r\bigr)$.
\end{proof}
This justifies the use of AUC as a scoring function in cases where the
forecaster is informed of the number of positive outcomes beforehand. This
means that the forecaster is able to use this information to rule out extreme
tail events that might otherwise have provided a windfall score. For example,
in the IJCNN Social Network Challenge by
Kaggle (\url{https://www.kaggle.com/c/socialNetwork}) competitors
were required to estimate 8960 binary outcomes (corresponding to
presence/absence of an edge), of which they were informed that exactly half
were positive.

\begin{theorem}
  \label{thm:indep}
  If the $Y_i$'s are mutually independent under all $P \in \Family$, then AUC
  is a strictly $R^*$-proper scoring function.
\end{theorem}
\begin{proof} 
  Note that if $y_i \neq y_j$, then $n_1(y) = 1 + n_1^{\neg(i,j)} (y)$, where
  $n_1^{\neg(i,j)} (y) = \sum_{k \ne i,j} y_k$, and similarly for $n_0$. Then
  \begin{equation*}
    \alpha_i(y) - \alpha_j(y)
    = \frac{y_i - y_j}{n_0(y) n_1(y)} 
    = \frac{y_i - y_j}{[1+ n_0^{\neg(i,j)}(y)][1+n_1^{\neg(i,j)} (y)]},
  \end{equation*}
  since if $y_i = y_j$, the  numerator is zero. Then by mutual independence,
  \begin{equation*}
    \E[\alpha_i(Y)] - \E[\alpha_j(Y)] = \left(\E[Y_i] - \E[Y_j] \right) 
    \E\left[\frac{1}{[1+ n_0^{\neg(i,j)}(Y)][1+n_1^{\neg(i,j)} (Y)]}\right].
  \end{equation*}
  As the latter expectation is strictly positive, it follows that $\E[\alpha_i(Y)] \leq
  \E[\alpha_j(Y)]$ if and only if $\E[Y_i] \leq \E[Y_j]$.
\end{proof}
As noted in section~\ref{sec:theoretical-auc}, mutual independence is a
somewhat unrealistic condition for scoring functions. Nevertheless, it can be
useful when combined with the following result.

\begin{theorem}
  \label{thm:latent}
  Let $\Family$ consist of distributions $P$ such that there is a latent variable
  $Z$ whereby
  \begin{enumerate}
  \item[(i)] for almost all $Z$, $\E_P[Y \mid Z]$ induces the same preordering
    as $\E_P[\alpha(Y)\mid Z]$, and
  \item[(ii)] this preordering is the same for almost all $Z$,
  \end{enumerate}
  then AUC is a strictly proper scoring function for $R^*$.
\end{theorem}
\begin{proof}
  Condition (i) implies that 
  \begin{equation*}
    \E_P[Y_i - Y_j \mid Z] \geq 0
    \quad\Leftrightarrow\quad
    \E_P[\alpha_i(Y) -\alpha_j(Y) \mid Z] \geq 0,
  \end{equation*}
  and by condition (ii) then,
  \begin{equation*}
    \E_P[\E[Y_i - Y_j \mid Z]] = \E_P[Y_i - Y_j] \geq 0
    \quad\Leftrightarrow\quad
    \E_P[\alpha_i(Y) - \alpha_j(Y)] \geq 0.
  \end{equation*}
\end{proof}
This provides a means for showing AUC is proper in more general contexts, by
combining it with one of the previous two theorems to satisfy condition
(i). For example, if $\theta$ is a parameter in a Bayesian model, conditional
on which the outcomes are independent (\eg a logistic regression model), then
AUC is proper for the predictive distributions if (ii) holds.

However these conditions can fail if there is significant uncertainty in the
ordering of the outcomes, which may arise in problems such as out-of-sample
prediction.
\begin{example}
  \label{exm:auc-counter-model}
  Suppose that there are two candidate models, $A$ and $B$,
  each weighted with probability 1/2, and the forecaster is to rank 100
  outcomes, of which 10 have a particular feature $U$ present. Suppose that
  the forecast probabilities are
  \begin{align*}
    \E[Y_i \mid U_i, A] &= 0.4 &     \E[Y_i \mid \neg U_i, A] &= 0.5 \\
    \E[Y_i \mid U_i, B] &= 0.95 &     \E[Y_i \mid \neg U_i, B] &= 0.9,
  \end{align*}
  and that outcomes are independent within each model. Then the resulting
  marginal probabilities are
  \begin{align*}
    \E[Y_i \mid U_i] &= 0.675 &     \E[Y_i \mid \neg U_i] &= 0.7 
  \end{align*}
  However using the induced ranking will result in an expected AUC of 0.496,
  whereas the opposite ranking will give an expected AUC of 0.504 (see
  supplementary material).
\end{example}

\section{Scoring functions for mappings}
\label{sec:map-scoring}

In many forecasting settings, each variable $Y_i$ has a corresponding
\emph{covariate} or \emph{feature} $X_i$ taking values in some measurable
space $\X_{\Cdot}$, which can be used to inform the prediction. In the case
where the forecaster is able to observe the covariates directly, we can assume
any relevant information is taken into account, and thus no additional
consideration is required.

However we can also consider the setting in which the forecaster does not
observe the covariates, but is instead required to provide some sort of
mapping from the covariate space $\X = (\X_{\Cdot})^n$ to the original prediction space
$\Predspace$ for $Y$ (we use the term \emph{mapping} so as to distinguish from
scoring functions). In other words, the forecaster is required to make a
prediction in the \emph{mapping prediction space}
\begin{equation*}
  \vec\Predspace = \{f : \X \to \Predspace\}.
\end{equation*}
Furthermore, any scoring function $s:\Y \times \Predspace \to \R$ has a
corresponding \emph{mapping form} $\vec s: (\X \times \Y) \times
\vec\Predspace \to \R$ which is simply $s$ evaluated using the mapping applied
to the observed covariates,
\begin{equation*}
  \vec s\bigl( (x,y) , f \bigr) = s\bigl( Y, f(X) \bigr).
\end{equation*}
Similarly, given any statistical functional $T:\Family \to \Predspace$, we can
define the corresponding \emph{mapping functional} $\vec T : \Family_{XY} \to
\vec\Predspace$ as the mapping of the conditional expectation
\begin{equation*}
  \vec T (P_{XY}) (x) = T(P_{Y \mid X = x}),
\end{equation*}
where $P_{Y \mid X=x}$ denotes the conditional distribution of $Y$ given $X
=x$ under $P$. That is, the optimal mapping should map each $x \in \X$
to the optimal prediction under the conditional distribution $P_{Y \mid X =
  x}$.

\begin{theorem}
\label{thm:map-opt}
Let $s$ be a $T$-proper scoring function for a family $\Family$, then $\vec s$
is a $\vec T$-proper scoring function for $\Family_{XY}$ if for each $P_{XY}
\in \Family_{XY}$, there exists a family of conditional distributions $\{P_{Y \mid
  X =x}\}_x$ which is a subset of $\Family$.
\end{theorem}
\begin{proof}
  The expected mapping score is
  \begin{equation*}
    \E \bigl[ \vec s\bigl( (x,y), f \bigr) \bigr]
    = \E \bigl[ \E\bigl[ s \bigl( Y, f(X) \bigr) \mid X \bigr] \bigr].
  \end{equation*}
  The inner expectation can be maximised for each value of $X \in \X$ by
  choosing $f(x) = \argmax_z \E[ s(Y,z) \mid X]$, which, as $s$ is
  $T$-proper, will be (an element of) $T(P_{Y \mid X = x})$.
\end{proof}

However we typically don't want to consider all possible mappings
$f:\X \to \Predspace$. Instead, we typically are only interested in mappings
that can be applied coordinate-wise,
\begin{equation*}
  f(x) = \bigl( f_{\Cdot}(x_1), \ldots, f_{\Cdot}(x_n) \bigr),
  \quad \text{where} \ 
  f_{\Cdot} : \X_{\Cdot} \to \R.
\end{equation*}
In other words, we constrain the mapping such that the forecast for each $Y_i$
depends only on its corresponding covariate $X_i$, and require that this
mapping be the same for all $i$.  Of course, we also need to constrain the
family of distributions to ensure that the marginal mapping is
coordinate-wise.
\begin{theorem}
  \label{thm:map-coord}
  Let $\vec\Family$ be the set of distributions for $(X,Y)$ such that
  \begin{enumerate}
  \item[(i)] $Y_i$ are conditionally independent of $X$ given $X_i$, and
  \item[(ii)] the distribution of $Y_i \mid X_i$ is the same for all $i$.
  \end{enumerate}
  Then for any $M$-proper scoring function $s$ for a family $\Family$,
  $\vec s$ is a $\vec M$-proper scoring function for the set of
  coordinate-wise mappings if the conditional distributions $P_{Y \mid X =x}$
  are in $\Family$. 
\end{theorem}
\begin{proof}
  By (i) we have that $\E[Y_i \mid X=x] = \E[Y_i \mid X_i=x_i]$, and by (ii)
  it follows that this quantity is the same for all $i$. Therefore the mapping
  $f(x) =\vec M(P_{Y\mid X=x})$ is coordinate-wise, which by
  Theorem~\ref{thm:map-opt}, implies that $\vec s$ is $\vec M$-proper.
\end{proof}

Consequently $\vec u$, the mapping form of $u$ is $\vec M$-proper for any
$\vec\Family$ satisfying (i) and (ii). For AUC to be $\vec M$-proper,
additional conditions are required, such as mutual independence of elements of
$Y$ conditional on $X$.

\section{Discussion}
\label{sec:discussion}

Although we have demonstrated that AUC is not generally a proper scoring
function, Examples~\ref{exm:auc-counter} and \ref{exm:auc-counter-model} both
exhibit quite extreme dependence between outcomes. Therefore, it might be
possible to establish a more relaxed criteria for establishing propriety of
AUC, for example, bounds on correlation or other measures of dependence.

We have also only considered the batch prediction setting where the forecaster
is required to provide the preordering for all $Y$ before any outcomes have
been observed. One alternative is a sequential framework, where at each point in time the
forecaster is required to provide a forecast for $Y_{t+1}$, having already
observed $Y_1,\ldots,Y_t$. In the ranking case, this requires the forecaster
to provide a total preorder $\precsim_{t+1}$ on $I_{t+1}$ that is compatible
with the one $\precsim_{t}$ provided on $I_t$. Unfortunately, rank-sum scoring
functions are essentially useless in this setting.

\begin{example}
  \label{exm:separation}
  Let $s$ be any rank-sum scoring rule of the form in \eqref{eq:ranksum},
  where $\sigma_i(y) = \sigma_j(y)$ if $y_i=y_j$, and $\sigma_i(y) \geq
  \sigma_j(y)$ if $y_i > y_j$ (both $u$ and the AUC satisfy this
  property). Then in the sequential setting, it is possible to maintain an
  optimal score by choosing $\precsim_{t+1}$ such that
  \begin{equation*}
    i \prec_{t+1} t+1 \prec_{t+1} j
    \quad \text{for all $i,j \leq t$ : $Y_i = 0$ and $Y_j=1$}.
  \end{equation*}
  By a straightforward application of induction, it is easy to see that such a
  sequence exists, and that it will maintain this ``perfect separation'', in
  that all $i$ where $Y_i =1$ will always be ranked above all $j$ where
  $Y_j=0$. Therefore, by Lemma~\ref{lem:sumvec}, this will result in the
  largest possible score (\ie an AUC of 1): note that unlike the
  previous sections, we refer to \emph{actual} score, not just the
  expected score.
\end{example}

In other words, it is possible to construct an optimal procedure with
absolutely no information whatsoever about the process of $Y_t$. This problem
will persist in the analogous mapping problem, where the forecaster is free to
choose the mapping $f_t : \X_{\Cdot} \to \R$ at each iteration.

\section*{Acknowledgements}
The author is grateful to the input of Matthew Parry, and the support from EPSRC fellowship EP/K005723/1.

\bibliographystyle{chicago}
\bibliography{refs}

\end{document}